\documentclass[12pt]{article}
\usepackage[left=2.5cm,right=2.5cm,top=2.5cm,bottom=2.5cm,a4paper]{geometry}

\usepackage[a4paper]{geometry}
\usepackage{graphicx}
\usepackage{microtype}
\usepackage{siunitx}
\usepackage{booktabs}
\usepackage{cleveref}
\usepackage{graphics}
\usepackage{graphicx}
\usepackage{epsfig}
\usepackage{amsmath,amsfonts,amssymb,amsthm}
\usepackage{listings}
\usepackage{paralist}
\usepackage{sectsty}
\numberwithin{equation}{section}
\date{}
\subsectionfont{\normalfont}

\newtheorem{theorem}{Theorem}[section]
\newtheorem{definition}[theorem]{Definition}

\newtheorem{corollary}[theorem]{Corollary}

\newtheorem{proposition}[theorem]{Proposition}

\newtheorem{open problem}[theorem]{Open problem}

\providecommand{\keywords}[1]{\textbf{\textit{Key words---}} #1}

\begin{document}
	
	\markboth{Donggyun Kim and Kyunghwan Song}
	{The inverses of tails of the Riemann zeta function}
	
	\title{The inverses of tails of the Riemann zeta function}
	
	\author{Donggyun Kim and Kyunghwan Song$^{*}$\\
		Department of Mathematics, Korea University, Seoul 02841, Republic of Korea}
	
	\maketitle
	
	\begin{abstract}
		We present some bounds of the inverses of tails of the Riemann zeta function on $0 < s < 1$ and compute the integer parts of the inverses of tails of the Riemann zeta function for $s=\frac{1}{2}, \frac{1}{3}$ and $\frac{1}{4}$.
		
		
	\end{abstract}

\maketitle
\section{Introduction}
The Riemann zeta function $\zeta(s)$ in the real variable $s$ was introduced by L. Euler \cite{Euler1737} in connection with questions about the distribution of prime numbers. Later B. Riemann \cite{Riemann (1859)} derived deeper results about a dual correspondence between the distribution of prime numbers and the complex zeros of $\zeta(s)$ in the complex variable $s$. In these developments, he asserted that all the non-trivial zeros of $\zeta(s)$ are on the line $\text{Re}(s) = \frac{1}{2}$, and this has been one of the most important unsolved problems in mathematics, called the Riemann hypothesis.  A vast amount of research on calculation of $\zeta(s)$ for the line $ \text{Re}(s) = \frac{1}{2}$ which is called the critical line, and on the strip $0 < \text{Re}(s) < 1$ which is called the critical strip, has been conducted using various methods \cite{Borwein2007}.

The {\em Riemann zeta function} and a {\em tail of the Riemann zeta function from $n$} for an integer $n \geq 1$ are defined, respectively, by for $\text{Re}(s) > 1$,
\begin{equation*}
	\zeta (s) = \sum_{k=1}^{\infty} \frac{1}{k^s} \quad \text{and} \quad \zeta_n (s) = \sum_{k=n}^{\infty} \frac{1}{k^s}
\end{equation*}
and for $0<\text{Re}(s)<1$,
\begin{equation*}
	\zeta (s) =\frac{1}{1-2^{1-s}}  \sum_{k=1}^{\infty} \frac{(-1)^{k+1}}{k^s} \quad \text{and} \quad \zeta_n (s) =\frac{1}{1-2^{1-s}}  \sum_{k=n}^{\infty} \frac{(-1)^{k+1}}{k^s}.
\end{equation*}

To understand the values of $\zeta(s)$, it would be helpful to understand the values of tails of $\zeta(s)$, for example, the integer parts of their inverses, $\left[\zeta_n (s)^{-1}\right]$, where $[x]$ denotes the greatest integer that is less than or equal to $x$.

Some values of $\left[\zeta_n (s)^{-1}\right]$ for small positive integers $s$ are known recently. L. Xin \cite{Xin (2016)} showed that for $s = 2$ and $3$,
\begin{equation*}
\left[\zeta_n (2)^{-1}\right]=n-1 \ \text{ and } \
\left[\zeta_n (3)^{-1}\right]=2n(n-1).
\end{equation*}
For $s=4$, L. Xin and L. Xiaoxue \cite{Xin (2017)} showed that
\begin{equation*}
 \left[\zeta_n (4)^{-1}\right]=
  3n^3 -5n^2 + 4n-1+\left[\frac{(2n+1)(n-1)}{4}\right]
\end{equation*}
for any integer $n \geq 2$ and H. Xu \cite{Xu (2016)} showed that for $s = 5$,
\begin{equation*}
\left[\zeta_n (5)^{-1}\right]= 4n^4 - 8n^3 + 9n^2 - 5n + \left[ \frac{(n+1)(n-2)}{3} \right]
\end{equation*}
for any integer $n \geq 4$.
W. Hwang and K. Song \cite{Hwang (2017)} provided an alternative proof of the case when $s = 5$ and a formula when $s = 6$ as follows. For an integer $n,$ write $n_{48}$ for the remainder when $n$ is divided by $48$, then
\begin{multline*}
\left[\zeta_n (6)^{-1}\right] = \\
 \begin{cases}
  5n^5 - \frac{25}{2} n^4 +\frac{75}{4} n^3 -\frac{125}{8} n^2 +\frac{185}{48}n - \frac{5 n_{48}}{48} - \left[\frac{35-5 n_{48}}{48}\right],     &\mbox{if~$n$~is even},  \\
  5n^5 - \frac{25}{2} n^4 +\frac{75}{4} n^3 -\frac{125}{8} n^2 +\frac{185}{48}n - \frac{5 n_{48}+18}{48} - \left[\frac{17-5 n_{48}}{48}\right],
    & \mbox{if~$n$~is odd}
  \end{cases}
\end{multline*}
for any integer $n \geq 829$. For the integer $s$ greater than 6, no such a formula is known.

There are other interesting results related to this theme such as  bounds  of $\zeta(3)$ in greater precision in \cite{Luo2003} and  \cite{Luo2003_2}.

We study the inverses of tails of the Riemann zeta function, $\zeta_n(s)^{-1}$, for $s$ on the critical strip $ 0 < s <1$. The following notation is needed to explain our results.

\begin{definition}
For any positive integer $n$ and real number $s$ with $0<s<1$, we define
\begin{align*}
A_{n,s} & = \left(\frac{1}{n^s} - \frac{1}{(n+1)^s}\right) + \left(\frac{1}{(n+2)^s} - \frac{1}{(n+3)^s}\right) + \cdots  \\
\intertext{and}
B_{n,s} & = \left(-\frac{1}{n^s} + \frac{1}{(n+1)^s}\right) + \left(-\frac{1}{(n+2)^s} + \frac{1}{(n+3)^s}\right) + \cdots .
\end{align*}
\end{definition}

Now the tail of the Riemann zeta function can be written as, for $0<s<1$,
\begin{equation}\label{eq:zetaAB}
\zeta_n(s) =
 \begin{cases}
    - \frac{1}{1-2^{1-s}} A_{n,s}, ~&~\mbox{if~$n$~is even}, \\
    - \frac{1}{1-2^{1-s}} B_{n,s}, & ~\mbox{if~$n$~is odd}.
 \end{cases}
\end{equation}

In this paper, we present the bounds of $A_{n,s}^{-1}$ and $B_{n,s}^{-1}$ hence the bounds of the inverses of tails of the Riemann zeta function,  $\zeta_n(s)^{-1}$, for $0<s<1$ in Section 2.1, and compute the values $\left[A_{n,s}^{-1}\right]$ and $\left[B_{n,s}^{-1}\right]$ hence the values of the inverses of tails of the Riemann zeta function, $\left[\frac{1}{1-2^{1-s}}\zeta_n(s) ^{-1}\right]$, for $s=\frac{1}{2},\, \frac{1}{3}$ and $\frac{1}{4}$ in Section 2.2.


\section{Main Results}
\subsection{The bounds of the inverses of $\zeta_n(s)$ for $0<s<1$}

In this section, we present the bounds of $A_{n,s}^{-1}$ and $B_{n,s}^{-1}$ in Definition \ref{eq:zetaAB}, hence the bounds of the inverses of tails of the Riemann zeta function,  $\zeta_n(s)^{-1}$, for $0<s<1$. 

\begin{proposition}\label{pro_2(n-1)_2n}
 Let $s$ be a real number with $0<s<1$. Then for any positive even number $n$,
 \begin{gather*}
  2(n-1)^s < A^{-1}_{n,s} < 2{n}^s
 \end{gather*}
 and for any positive odd number $n$,
 \begin{gather*}
  -2{n}^s < B^{-1}_{n,s} < -2(n-1)^s.
 \end{gather*}
\end{proposition}

\begin{proof}
Let $n$ be a positive even number. For every positive integer $k$, it is easy to see that
\begin{multline*}
\left( \frac{1}{(n+1+2k)^s} - \frac{1}{(n+2+2k)^s} \right) \\
  < \left( \frac{1}{(n+2k)^s} - \frac{1}{(n+1+2k)^s} \right) \\
    < \left( \frac{1}{(n-1+2k)^s} - \frac{1}{(n+2k)^s} \right).
\end{multline*}
The summations of each term over $k$ give
\begin{gather*}
A_{n+1,s} < A_{n,s} < A_{n-1,s}
\end{gather*}
and
\begin{gather*}
\frac{1}{2} (A_{n+1,s} + A_{n,s}) < A_{n,s} < \frac{1}{2} (A_{n-1,s} + A_{n,s}).
\end{gather*}
Therefore, we have
\begin{gather*}
\frac{1}{2n^s} < A_{n,s} < \frac{1}{2(n-1)^s},
\end{gather*}
which gives the first statement.

The second statement can be shown similarly.
\end{proof}

Since every proof of the case when $n$ is an odd number is analogous to that of the case when $n$ is an even number, we omit all the proofs of the odd number cases in this paper.

Now we find tighter bounds for $A^{-1}_{n,s}$ and $B^{-1}_{n,s}$.

\begin{proposition}\label{thm_0to1_lower}
 Let $s$ be a real number with $0<s<1$. Then for any positive even number $n$,
	\begin{equation*}
	2\left(n-\frac{1}{2}\right)^{s} < A^{-1}_{n,s}
	\end{equation*}
 and for any positive odd number $n$,
 \begin{equation*}
  B^{-1}_{n,s} < -2\left(n-\frac{1}{2}\right)^{s} .
 \end{equation*}
\end{proposition}

\begin{proof}
 Let $n$ be a positive even number. We will show that
	\begin{gather*}
	A_{n,s} < \frac{1}{2\left(n-\frac{1}{2}\right)^s}.
\end{gather*}

Rewriting each of the both sides as a series,
 \begin{align*}
 	A_{n,s} & = \sum_{k=\frac{n}{2}}^{\infty} \left(\frac{1}{(2k)^{s}} - \frac{1}{(2k+1)^{s}}\right) \ \ \text{and} \\
    \frac{1}{2(n-\frac{1}{2})^s} & = \sum_{k=\frac{n}{2}}^{\infty} \left( \frac{1}{2(2k-\frac{1}{2})^s} - \frac{1}{2( 2k + \frac{3}{2})^s} \right),
 \end{align*}
we will show that for any positive integer $k$,
\begin{gather*}
\frac{1}{(2k)^s} - \frac{1}{(2k+1)^s} < \frac{1}{2\left(2k - \frac{1}{2}\right)^s} - \frac{1}{2\left(2k+\frac{3}{2}\right)^s}.
\end{gather*}

For this, we let
\[ f(x) = \left(\frac{1}{2(2x-\frac{1}{2})^s} - \frac{1}{2(2x+\frac{3}{2})^s}\right) - \left(\frac{1}{(2x)^s} - \frac{1}{(2x+1)^s} \right)
\]
 and will show that $f(x)$ is positive for $x \ge 1$ and $0<s<1$. With 
\[ g(x) = \frac{1}{2(2x-\frac{1}{2})^s} + \frac{1}{2(2x+\frac{1}{2})^s} - \frac{1}{(2x)^s},
\]
we have $f(x) = g(x) -g(x+\frac{1}{2})$. Consider the derivative of $g(x)$, 
\[ g'(x) =-2s\left( \frac{1}{2(2x-\frac{1}{2})^{s+1}} + \frac{1}{2(2x+\frac{1}{2})^{s+1}} - \frac{1}{(2x)^{s+1}}\right).
\]
Since the function $\frac{1}{x^{s+1}}$ is convex, we obtain that
\[ \frac{1}{2(2x-\frac{1}{2})^{s+1}} + \frac{1}{2(2x+\frac{1}{2})^{s+1}} - \frac{1}{(2x)^{s+1}} \geq 0 \]
and therefore $g'(x)$ is negative, that is, $g(x)$ is decreasing. We conclude that  $f(x)$ is positive which gives the statement.
\end{proof}

\begin{proposition}\label{thm_0to1_upper}
Let $s$ be a real number with $0<s<1$. Then for any positive even number $n$,
	\begin{equation*}
	A^{-1}_{n,s} < 2\left(n-\frac{1}{4}\right)^{s}
	\end{equation*}
and for any positive odd number $n$,
\begin{equation*}
 -2\left(n-\frac{1}{4}\right)^{s}< B^{-1}_{n,s}.
\end{equation*}
\end{proposition}

\begin{proof}
Let $n$ be a positive even
  number. We will show that
\begin{gather*}
	\frac{1}{2\left(n-\frac{1}{4}\right)^s} < A_{n,s}.
	\end{gather*}
	
Rewriting each of the both sides as a series,
\begin{align*}
	A_{n,s} & = \sum_{k=\frac{n}{2}}^{\infty} \left( \frac{1}{(2k)^s} - \frac{1}{(2k+1)^s} \right) \ \ \text{and} \\
	\frac{1}{2(n-\frac{1}{4})^s} & = \sum_{k=\frac{n}{2}}^{\infty} \left( \frac{1}{2(2k-\frac{1}{4})^s} - \frac{1}{2(2k+\frac{7}{4})^s} \right),
\end{align*}
we need to show that for any positive integer $k$,
\begin{gather*}
	\frac{1}{2(2k-\frac{1}{4})^s} - \frac{1}{2(2k+\frac{7}{4})^s} < \frac{1}{(2k)^s} - \frac{1}{(2k+1)^s}.
\end{gather*}

For this, we let
\[ f(x) = \left( \frac{1}{(2x)^s} - \frac{1}{(2x+1)^s} \right) - \left( \frac{1}{2(2x - \frac{1}{4})^s} - \frac{1}{2(2x+\frac{7}{4})^s} \right).
\]
We check that $f(1)>0$ and now we will show that $f(x)$ is positive for $x \geq 2$ and $0 < s < 1$.
With
\[ g(x) = \frac{1}{(2x)^s} - \left( \frac{1}{2(2x-\frac{1}{4})^s} + \frac{1}{2(2x+\frac{3}{4})^s} \right),
\]
we have $f(x) = g(x) - g(x+\frac{1}{2})$, so we only need to show that $g(x)$ is decreasing. Consider the derivative of $g(x)$,
\begin{align*}
	g'(x) & = s\left( - \frac{2}{(2x)^{s+1}} + \left( \frac{1}{\left(2x-\frac{1}{4}\right)^{s+1}} +  \frac{1}{\left(2x+\frac{3}{4}\right)^{s+1}}  \right) \right) \\
	& = s\left(\left( \frac{1}{\left(2x-\frac{1}{4}\right)^{s+1}} -  \frac{1}{(2x)^{s+1}}  \right) - \left( \frac{1}{(2x)^{s+1}} -  \frac{1}{\left(2x+\frac{3}{4}\right)^{s+1}}  \right)\right).
\end{align*}

Since the function $\frac{1}{x^{s+1}}$ is decreasing and convex, by comparing slopes at $(2x-\frac{1}{4})$ and $(2x+\frac{3}{4})$, we obtain
	\begin{gather*}
	\frac{1}{\left(2x-\frac{1}{4}\right)^{s+1}} - \frac{1}{(2x)^{s+1}} < \frac{1}{4} (s+1) \frac{1}{\left(2x-\frac{1}{4}\right)^{s+2}}
	\end{gather*}
	and
	\begin{gather*}
	\frac{1}{(2x)^{s+1}} - \frac{1}{\left(2x+\frac{3}{4}\right)^{s+1}} > \frac{1}{4} (s+1) \frac{3}{\left(2x+\frac{3}{4}\right)^{s+2}}.
	\end{gather*}
Therefore,
\begin{gather*}
 g'(x) < \frac{1}{4} s (s+1) \left( \frac{1}{\left(2x-\frac{1}{4}\right)^{s+2}} - \frac{3}{\left(2x+\frac{3}{4}\right)^{s+2}} \right).
\end{gather*}
Consider $h(x,s):= \frac{1}{3} \left( \frac{2x+3/4}{2x-1/4} \right)^{s+2}$ which is the ratio of two terms on the right-hand side of the above expression.  We check that $h(x,s) < 1$ for $x \geq 2$ and $0<s<1$. Since  $h(2,1) = 6859/10125$ and $\lim_{x \rightarrow \infty} h(x,s) = \frac{1}{3}$ for $0<s<1$, we obtain that $g'(x)$ is negative and therefore $g(x)$ is decreasing which gives the statement.
\end{proof}

We combine the results of Proposition \ref{thm_0to1_lower} and Proposition \ref{thm_0to1_upper}.

\begin{theorem}\label{cor_AB}
Let $s$ be a real number with $0<s<1$. Then for any positive even number $n$,
\begin{gather*}
2\left(n-\frac{1}{2}\right)^s < A^{-1}_{n,s} < 2\left(n - \frac{1}{4}\right)^s
\end{gather*}
and for any positive odd number $n$,
\begin{gather*}
-2\left(n-\frac{1}{4}\right)^s < B^{-1}_{n,s} < -2\left(n - \frac{1}{2}\right)^s.
\end{gather*}
\end{theorem}

We express these bounds in terms of $\zeta_n(s)$ using expression (\ref{eq:zetaAB}).

\begin{corollary}\label{cor_zeta}
 Let $s$ be a real number with $0<s<1$. Then for any positive even number $n$,
\begin{gather*}
2(1-2^{1-s})\left(n-\frac{1}{4}\right)^s < \zeta_n(s)^{-1} < 2(1-2^{1-s})\left(n - \frac{1}{2}\right)^s
\end{gather*}
and for any positive odd number $n$,
\begin{gather*}
-2(1-2^{1-s})\left(n-\frac{1}{2}\right)^s < \zeta_n(s)^{-1} < -2(1-2^{1-s})\left(n - \frac{1}{4}\right)^s.
\end{gather*}	
\end{corollary}

Furthermore, we have tighter bounds of $A^{-1}_{n,s}$ and $B^{-1}_{n,s}$ for a sufficiently large number $n$.

\begin{theorem}\label{thm_0to1_tight}
For any positive number $\epsilon$ and any real number $s$ with $0<s<1$,
\begin{gather*}
2\left(n-\frac{1}{2}\right)^s < A^{-1}_{n,s} < 2\left(n - \frac{1}{2} + \epsilon\right)^s
\end{gather*}
for a sufficiently large even number $n$ and
\begin{gather*}
-2\left(n-\frac{1}{2} + \epsilon \right)^s < B^{-1}_{n,s} \leq -2\left(n - \frac{1}{2} \right)^s
\end{gather*}
for a sufficiently large odd number $n$.
\end{theorem}

\begin{proof}
From Theorem \ref{cor_AB}, it suffices to show that for a sufficiently large even number $n$,
\begin{gather*}
\frac{1}{2\left(n-\frac{1}{2} + \epsilon\right)^{s}} < A_{n,s}.
\end{gather*}

Rewriting each of the both sides as a series,
\begin{align*}
 A_{n,s} & = \sum_{k=\frac{n}{2}}^{\infty} \left( \frac{1}{(2k)^s} - \frac{1}{(2k+1)^s}\right) \text{and} \\
\frac{1}{2\left(n - \frac{1}{2} + \epsilon\right)^s} & = \sum_{k=\frac{n}{2}}^{\infty} \left( \frac{1}{2\left(2k-\frac{1}{2} + \epsilon\right)^s} -  \frac{1}{2\left(2k+\frac{3}{2} + \epsilon\right)^s}  \right),
\end{align*}
we need to show that for a sufficiently large even number $n$ and every integer $k \geq \frac{n}{2}$,
\begin{gather*}
\frac{1}{2\left(2k - \frac{1}{2} + \epsilon\right)^s} - \frac{1}{2\left(2k + \frac{3}{2} + \epsilon\right)^s} < \frac{1}{(2k)^s} - \frac{1}{(2k+1)^s}.
\end{gather*}

For this, let
\[ f(x) = \left(\frac{1}{(2x)^s} - \frac{1}{(2x+1)^s}\right) - \left(\frac{1}{2(2x-\frac{1}{2} + \epsilon)^s} - \frac{1}{2(2x+\frac{3}{2} + \epsilon)^s}\right)
\]
and will show that $f(x)$ is positive for $x \geq x_0$ where $x_0$ is a sufficiently large number.
With
\[ g(x)= \frac{1}{(2x)^s} - \left( \frac{1}{2(2x-\frac{1}{2} + \epsilon)^s} + \frac{1}{2(2x+\frac{1}{2} + \epsilon)^s} \right),
\]
we have that $f(x) = g(x) - g(x+\frac{1}{2})$, so we only need to show that $g(x)$ is decreasing. Consider the derivative of $g(x)$, 
\begin{align*}
g'(x) & = s\left( -\frac{2}{(2x)^{s+1}} + \frac{1}{\left( 2x - \frac{1}{2} + \epsilon \right)^{s+1}} + \frac{1}{\left( 2x + \frac{1}{2} + \epsilon \right)^{s+1}}\right) \\
& = s \left( \left( \frac{1}{\left(2x - \frac{1}{2} + \epsilon\right)^{s+1}} - \frac{1}{(2x)^{s+1}} \right) - \left( \frac{1}{(2x)^{s+1}} - \frac{1}{\left(2x + \frac{1}{2} + \epsilon \right)^{s+1}} \right)  \right).
\end{align*}
Since $\frac{1}{x^{s+1}}$ is decreasing and convex, by comparing slopes at $(2x - \frac{1}{2} + \epsilon )$ and $(2x + \frac{1}{2} + \epsilon)$,  we obtain
\begin{gather*}
\frac{1}{\left(2x - \frac{1}{2} + \epsilon \right)^{s+1}} - \frac{1}{(2x)^{s+1}} <  (s+1) \frac{\frac{1}{2} - \epsilon}{\left(2x - \frac{1}{2} + \epsilon\right)^{s+2}}
\end{gather*}
and
\begin{gather*}
\frac{1}{(2x)^{s+1}} - \frac{1}{\left(2x+\frac{1}{2} + \epsilon\right)^{s+1}} >  (s+1) \frac{\frac{1}{2} + \epsilon}{\left(2x+\frac{1}{2} + \epsilon\right)^{s+2}}.
\end{gather*}
Therefore
\begin{gather*}
g'(x) < s(s+1) \left( \frac{\frac{1}{2} - \epsilon}{\left(2x - \frac{1}{2} + \epsilon\right)^{s+2}} - \frac{\frac{1}{2} + \epsilon}{\left(2x + \frac{1}{2} + \epsilon\right)^{s+2}}   \right).
\end{gather*}
Consider $h(x):= \frac{\frac{1}{2} - \epsilon}{\frac{1}{2} + \epsilon} \left( \frac{2x+\frac{1}{2} + \epsilon}{2x - \frac{1}{2} + \epsilon} \right)^{s+2}$ which is the ratio of two terms on the right-hand side of the above expression. We need to show that $h(x) < 1$ for every $x > x_0$ where $x_0$ is a sufficiently large number. We check that
\begin{gather*}
h(x) < 1 \iff \frac{2x+\frac{1}{2} + \epsilon}{2x - \frac{1}{2} + \epsilon} < \left( \frac{\frac{1}{2} + \epsilon}{\frac{1}{2} - \epsilon} \right)^{\frac{1}{s+2}}.
\end{gather*}
For any $\epsilon > 0$ and $0 < s < 1$, we have that $1 < \left( \frac{\frac{1}{2} + \epsilon}{\frac{1}{2} - \epsilon} \right)^{1/(s+2)}$ and $\frac{2x + \frac{1}{2} + \epsilon}{2x - \frac{1}{2} + \epsilon}$ is larger than $1$, decreasing and converges to $1$ as $x$ goes to infinite, so there is $x_0$ so that for every $x > x_0$, $h(x) < 1$. Therefore the proof is complete.
\end{proof}

We express these bounds in terms of $\zeta_n(s)$ using expression (\ref{eq:zetaAB}).

\begin{corollary}
For any positive number $\epsilon$ and any real number $s$ with $0<s<1$,
\begin{gather*}
 2(1-2^{1-s})\left(n-\frac{1}{2}+\epsilon \right)^s < \zeta_n(s)^{-1} < 2(1-2^{1-s})\left(n - \frac{1}{2}\right)^s
\end{gather*}
for a sufficiently large even number $n$ and
\begin{gather*}
-2(1-2^{1-s})\left(n-\frac{1}{2}\right)^s < \zeta_n(s)^{-1} < -2(1-2^{1-s})\left(n - \frac{1}{2}+\epsilon \right)^s
\end{gather*}
for a sufficiently large odd number $n$.
\end{corollary}


\subsection{The value of the inverse of $\zeta_n(s)$ for $s=\frac{1}{2}, \frac{1}{3}$ and $\frac{1}{4}$ } \label{sec_Expected}

We study firstly the value of the inverse of $\zeta_n(\frac{1}{2})$, where $\zeta_n(\frac{1}{2})$ is the tail of the Riemann zeta function from $n$ at $s=\frac{1}{2}$.

\begin{theorem}\label{thm_1/2}
 For any positive even number $n$,
 \begin{gather*}
  [A^{-1}_{n, 1/2}] = \left[2\left(n-\frac{1}{2}\right)^{1/2}\right]
 \end{gather*}
 and for any positive odd number $n$,
 \begin{gather*}
 [B^{-1}_{n, 1/2}] = \left[-2\left(n-\frac{1}{2}\right)^{1/2}\right].
 \end{gather*}
\end{theorem}

\begin{proof}
 Let $n$ be a positive even number. By Theorem \ref{cor_AB}, we have that
\begin{gather*}
 2\left( n - \frac{1}{2} \right)^{1/2} < A^{-1}_{n,1/2} < 2\left( n - \frac{1}{4} \right)^{1/2}.
\end{gather*}
Note that $2(n - \frac{1}{4})^{1/2} - 2(n - \frac{1}{2})^{1/2} < 1$ for $n \geq 2$, and it implies that there is at most one integer in the open interval from $2(n-\frac{1}{2})^{1/2}$ to $2(n-\frac{1}{4})^{1/2}$.
Suppose that there is an integer $h$ in the open interval, i.e.,
\[ 2\left(n - \frac{1}{2}\right)^{1/2} < h < 2\left(n - \frac{1}{4}\right)^{1/2} \ \ \text{or} \ \ \ 4n-2 < h^2 < 4n-1.
\]
There is, however, no integer in the open interval from $4n-2$ to $4n-1$, therefore such an integer $h$ does not exist. This gives the statement.
\end{proof}

We express this result in terms of $\zeta_n(s)$ using expression (\ref{eq:zetaAB}).

\begin{corollary}\label{cor_1/2}
For any positive integer $n$,
\begin{gather*}
\left[ \frac{1}{1-2^{1/2}} \zeta_n\left(\frac{1}{2}\right)^{-1}  \right]  = \left[(-1)^{n+1}  2\left(n-\frac{1}{2}\right)^{1/2}\right].
\end{gather*}
\end{corollary}


We study secondly the value of the inverse of $\zeta_n(\frac{1}{3})$, where $\zeta_n(\frac{1}{3})$ is the tail of the Riemann zeta function from $n$ at $s=\frac{1}{3}$.

\begin{theorem}\label{thm_1/3}
	For any positive even number $n$,
	\begin{gather*}
	[A^{-1}_{n, 1/3}] = \left[2\left(n-\frac{1}{2}\right)^{1/3}\right]
	\end{gather*}
and for any positive odd number $n$,
\begin{gather*}
[B^{-1}_{n, 1/3}] = \left[-2\left(n-\frac{1}{2}\right)^{1/3}\right].
\end{gather*}
\end{theorem}

\begin{proof}
Let $n$ be a positive even number. By Theorem \ref{cor_AB}, we have that
\begin{gather*}
 2\left( n - \frac{1}{2} \right)^{1/3} < A^{-1}_{n,1/3} < 2\left( n - \frac{1}{4} \right)^{1/3}.
\end{gather*}
Note that $2(n - \frac{1}{4})^{1/3} - 2(n - \frac{1}{2})^{1/3} < 1$ for $n \geq 2$, and it implies that there is at most one integer in the open interval from $2(n-\frac{1}{2})^{1/3}$ to $2(n-\frac{1}{4})^{1/3}$. Suppose that there is an integer $h$ in the open interval, i.e.,
\[ 2(n - \frac{1}{2})^{1/3} < h < 2(n - \frac{1}{4})^{1/3} \ \ \text{or}\ \ 8n - 4 < h^3 < 8n - 2.
\]
This shows that the integer $h$ is of the form $h = 2(n-\frac{3}{8})^{1/3}$. If we show $A^{-1}_{n,1/3} < 2(n-\frac{3}{8})^{1/3}$ or equivalently, $\frac{1}{2(n-\frac{3}{8})^{1/3}} < A_{n,1/3}$, then our proof will be done. Let us rewrite 
\begin{align*}
  A_{n,1/3} & = \sum_{k=\frac{n}{2}}^{\infty} \left( \frac{1}{(2k)^{1/3}} - \frac{1}{(2k+1)^{1/3}} \right) \ \text{and} \\
  \frac{1}{2(n-\frac{3}{8})^{1/3}} & = \sum_{k=\frac{n}{2}}^{\infty} \left( \frac{1}{2(2k-\frac{3}{8})^{1/3}} - \frac{1}{2(2k+\frac{13}{8})^{1/3}} \right).
\end{align*}
Now it suffices to show that for any positive integer $k$,
\begin{gather*}
\frac{1}{2\left(2k - \frac{3}{8}\right)^{1/3}} - \frac{1}{2\left(2k + \frac{13}{8}\right)^{1/3}} < \frac{1}{(2k)^{1/3}} - \frac{1}{2\left(2k +1\right)^{1/3}}.
\end{gather*}
For this, we let
\[ f(x) = \left( \frac{1}{(2x)^{1/3}} - \frac{1}{(2x+1)^{1/3}} \right) - \left( \frac{1}{2(2x - \frac{3}{8})^{1/3}} - \frac{1}{2(2x+\frac{13}{8})^{1/3}} \right)
\]
and we will show that $f(x)$ is positive for any positive integer $x$.

We check that $f(1) = 0.00053\cdots $ and $f(2) = 0.00081\cdots$, so it suffices to show $f(x) > 0$ for $x \geq 3$.
With
\[ g(x) = \frac{1}{(2x)^{1/3}} - \left( \frac{1}{2(2x - \frac{3}{8})^{1/3}} + \frac{1}{2(2x + \frac{5}{8})^{1/3}} \right),
\]
we have that $f(x) = g(x) - g(x+\frac{1}{2})$,
so we only need to show that $g(x)$ is decreasing for $x \ge 3$. Consider the derivative of $g(x)$,
\begin{align*}
g'(x) & = \frac{1}{3} \left( -\frac{2}{(2x)^{4/3}} +  \frac{1}{\left(2x - \frac{3}{8}\right)^{4/3}} + \frac{1}{\left(2x + \frac{5}{8}\right)^{4/3}}  \right)  \\
      & = \frac{1}{3} \left(\left( \frac{1}{\left(2x -\frac{3}{8}\right)^{4/3}} - \frac{1}{(2x)^{4/3}} \right) - \left( \frac{1}{(2x)^{4/3}} - \frac{1}{\left(2x + \frac{5}{8}\right)^{4/3}} \right)\right) .
\end{align*}
Since $\frac{1}{x^{4/3}}$ is decreasing and convex, by comparing slopes at $(2x - \frac{3}{8})$ and $(2x + \frac{5}{8})$, we obtain
\begin{gather*}
\frac{1}{\left( 2x - \frac{3}{8} \right)^{4/3}} - \frac{1}{(2x)^{4/3}} < 2\cdot \frac{3}{16} \cdot \frac{4}{3} \cdot \frac{1}{\left(2x - \frac{3}{8}\right)^{7/3}}
\end{gather*}
and
\begin{gather*}
\frac{1}{(2x)^{4/3}} - \frac{1}{\left(2x + \frac{5}{8}\right)^{4/3}} > 2\cdot \frac{5}{16} \cdot \frac{4}{3} \cdot \frac{1}{\left(2x + \frac{5}{8}\right)^{7/3}}.
\end{gather*}
Therefore
\begin{gather*}
g'(x) < \frac{1}{18}\left( \frac{3}{\left(2x - \frac{3}{8}\right)^{7/3}} - \frac{5}{\left(2x + \frac{5}{8}\right)^{7/3}} \right).
\end{gather*}
Consider $h(x):= \frac{3}{5} \left( \frac{2x + 5/8}{2x - 3/8} \right)^{7/3}$ which is the ratio of two terms of the right-hand side of the above expression. We check that $h(x) < 1$ for $x \geq 3$ because $h(3) = 0.87\cdots$ and $\lim_{x \rightarrow \infty} h(x) = \frac{3}{5}$ and $h'(x) < 0$ for $x \geq 3$. Hence we obtain that $g'(x)$ is negative and so $g(x)$ is decreasing for $x \geq 3$ which proves the statement.
\end{proof}

We express this result in terms of $\zeta_n(s)$ using expression (\ref{eq:zetaAB}).

\begin{corollary}\label{cor_1/3}
  For any positive integer $n$,
	\begin{gather*}
	\left[ \frac{1}{1-2^{2/3}} \zeta_n\left(\frac{1}{3}\right)^{-1}  \right]  = \left[(-1)^{n+1} 2\left(n-\frac{1}{2}\right)^{1/3}\right].
	\end{gather*}
\end{corollary}


We study lastly the value of the inverse of $\zeta_n(\frac{1}{4})$, which is the tail of the Riemann zeta function from $n$ at $s=\frac{1}{4}$.

\begin{theorem}\label{thm_1/4}
 For any positive even number $n$,
 \begin{gather*}
  [A^{-1}_{n, 1/4}] = \left[2\left(n-\frac{1}{2}\right)^{1/4}\right]
 \end{gather*}
 and for any positive odd number $n$,
 \begin{gather*}
 [B^{-1}_{n, 1/4}] = \left[-2\left(n-\frac{1}{2}\right)^{1/4}\right].
 \end{gather*}
\end{theorem}

\begin{proof}
 Let $n$ be a positive even number. By Theorem \ref{cor_AB}, we have that
 \begin{gather*}
  2\left( n - \frac{1}{2} \right)^{1/4} < A^{-1}_{n,1/4} < 2\left( n - \frac{1}{4} \right)^{1/4}.
 \end{gather*}
 Note that $2(n - \frac{1}{4})^{1/4} - 2(n - \frac{1}{2})^{1/4} < 1$ for $n \geq 2$ and it implies that there is at most one integer in the open interval from $2(n-\frac{1}{2})^{1/4}$ to $2(n-\frac{1}{4})^{1/4}$. Suppose that there is an integer $h$ in the open interval, i.e.,
 \[ 2(n-\frac{1}{2})^{1/4} < h < 2(n - \frac{1}{4})^{1/4} \ \ \text{or} \ \ 16n - 8 < h^4 < 16n - 4.
 \]
 This shows that the integer $h^4$ is one of the form $16n - 7, 16n - 6$ or $16n - 5$. For any integer $h$, however $h^4 \equiv 0$ or $1 \pmod{16}$, hence such an integer $h$ does not exist. Therefore this gives the statement.
\end{proof}

We express this result in terms of $\zeta_n(s)$ using expression (\ref{eq:zetaAB}).

\begin{corollary}\label{cor_1/4}
  For any positive integer $n$,
	\begin{gather*}
	\left[ \frac{1}{1-2^{3/4}} \zeta_n\left(\frac{1}{4}\right)^{-1}  \right]  = \left[(-1)^{n+1} 2\left(n-\frac{1}{2}\right)^{1/4}\right].
	\end{gather*}
\end{corollary}

We express the results of Theorems \ref{thm_1/2}, \ref{thm_1/3} and \ref{thm_1/4} in a single statement.

\begin{theorem}
  For $s = \frac{1}{2},\, \frac{1}{3}$, or $\frac{1}{4}$,
and for any positive even number $n$,
	\begin{gather*}
	[A^{-1}_{n, s}] = \left[2\left(n-\frac{1}{2}\right)^s\right]
	\end{gather*}
and for any positive odd number $n$,
\begin{gather*}
[B^{-1}_{n, s}] = \left[-2\left(n-\frac{1}{2}\right)^s\right].
\end{gather*}
\end{theorem}

We express the results of Corollary \ref{cor_1/2}, \ref{cor_1/3} and \ref{cor_1/4} in a single statement.

\begin{corollary}
  For any positive integer $n$ and $s = \frac{1}{2},\, \frac{1}{3}$, or $\frac{1}{4}$,
  \begin{gather*}
    \left[ \frac{1}{1-2^{1-s}} \zeta_n\left(s\right)^{-1}  \right]  = \left[(-1)^{n+1} 2\left(n-\frac{1}{2}\right)^{s}\right].
  \end{gather*}
\end{corollary}

\section{Conclusion}
In this paper, we present the bounds of $A_{n,s}^{-1}$ and $B_{n,s}^{-1}$ hence the bounds of the inverses of tails of the Riemann zeta function,  $\zeta_n(s)^{-1}$, for $0<s<1$, and compute the values $\left[A_{n,s}^{-1}\right]$ and $\left[B_{n,s}^{-1}\right]$ hence the values of the inverses of tails of the Riemann zeta function, $\left[\frac{1}{1-2^{1-s}}\zeta_n(s) ^{-1}\right]$, for $s=\frac{1}{2},\, \frac{1}{3}$ and $\frac{1}{4}$.
For other values of $s$, for example $s=\frac{1}{5}$ or $\frac{2}{3} $, etc, the values of $A_{n,s}$ and $B_{n,s}$ don't seem to have simple expressions.

\section*{Authors' Contributions}
All authors equally contributed to the manuscript.
\section*{Acknowledgements}
The authors would like to express their thanks to the referees and the editors for their helpful comments and advice.



\end{document}